\documentclass{amsart}
\usepackage{amsmath}
\usepackage{amsthm}
\usepackage{amssymb}
\usepackage[utf8]{inputenc}
\usepackage[english]{babel}
\usepackage{hyperref}
\usepackage[autostyle]{csquotes}
\usepackage[shortlabels]{enumitem}
\usepackage{xspace}
\usepackage{mathtools}
\usepackage{xcolor}

\theoremstyle{definition}
\newtheorem{defn}{Definition}[section]
\newtheorem{example}[defn]{Example}
\newtheorem{fact}[defn]{Fact}

\newtheorem{remark}[defn]{Remark}
\theoremstyle{plain}
\newtheorem{lemma}[defn]{Lemma}
\newtheorem{proposition}[defn]{Proposition}
\newtheorem{theorem}[defn]{Theorem}
\newtheorem{corollary}[defn]{Corollary}
\newcommand{\G}{\mathbb{G}}
\newcommand{\GG}{\mathfrak{G}}

\DeclareMathOperator{\diam}{diam}
\DeclareMathOperator{\Gen}{Gen}

\begin{document}
\title[On the universal valued Abelian groups of Niemiec]{The universal valued Abelian groups of Niemiec are L\'{e}vy, strongly exotic, extremely amenable, and $\G_r(0)$ is generically monothetic}
\author{Alessandro Codenotti}
\address{Dipartimento di Matematica, Universit\`{a} di Bologna, Piazza di
Porta S. Donato, 5, 40126 Bologna,\ Italy}
\email{alessandro.codenotti@unibo.it}
\urladdr{https://sites.google.com/view/alessandro-codenotti/}
\author{Martino Lupini}
\address{Dipartimento di Matematica, Universit\`{a} di Bologna, Piazza di
Porta S. Donato, 5, 40126 Bologna,\ Italy}
\email{martino.lupini@unibo.it}
\urladdr{http://www.lupini.org/}
\thanks{The authors were partially supported by the Starting Grant 101077154
\textquotedblleft Definable Algebraic Topology\textquotedblright\ from the
European Research Council, the Universit\'{e} Paris Cit\'{e}, the Gruppo
Nazionale per le Strutture Algebriche, Geometriche e le loro Applicazioni
(GNSAGA) of the Istituto Nazionale di Alta Matematica (INDAM), and the
University of Bologna. }
\subjclass[2020]{Primary 22A05, 03E15; Secondary 20K99, 43A07, 54H11}
\keywords{Extreme amenability; valued Abelian group; invariant metric;
generic property; Fra\"{\i}ss\'{e} limit; Polish group; universal Abelian
group}
\date{\today }

\begin{abstract}
We prove various properties of the universal valued Abelian groups $\mathbb{G}_r(N)$ constructed by Niemiec, for $r\in\{1,\infty\}$ and $N\in\{0,2,3,\ldots\}$. First we show that the completion of $\Gamma_0=\bigoplus_{n\in\mathbb{N}}\mathbb{Q}/\mathbb{Z}$ (for $N=0$) or $\Gamma_N=\bigoplus_{n\in\mathbb{N}}\mathbb{Z}/N\mathbb{Z}$ (for $N\geq 2$) with respect to a generic invariant metric, bounded by $1$ when $r=1$, is isometrically group-isomorphic to $\mathbb{G}_r(N)$, recovering a result of Doucha in the case when $r=\infty $ and $N=0$. This confirms an expectation of Doucha for the group $\Gamma_N$. We combine this genericity result with a criterion of Melleray and Tsankov concerning the extreme amenability of the generic completion of a countable group to obtain the extreme amenability of $\G_r(N)$. We then establish the strictly stronger properties that these groups are L\'{e}vy and strongly exotic. We conclude by showing that $\G_r(0)$ is also monothetic, from which we obtain the existence of a monothetic group structure on the Urysohn sphere, giving a bounded version of a result by Cameron and Vershik and answering a question of Niemiec.
\end{abstract}

\maketitle

\section{Introduction}
For $r\in \{1,\infty \}$ and $N\in
\{0,2,3,\ldots \}$, Niemiec \cite{Niemiec} constructed a valued group $\mathbb{G}_{r}(N)$
universal for separable valued Abelian groups of class $\mathcal{O}_{0}$
when $N=0$ and for separable valued Abelian groups of exponent dividing $N$ when $%
N\geq 2$; when $r=1$ its value is bounded by $1$ and it is universal for the corresponding classes of valued groups whose values are also bounded by $1$. The groups $\mathbb{G}_{r}(N)
$ are counterparts, in the category of valued Abelian groups, of the Urysohn
space \cite{urysohn1927universel} in the category of metric spaces and of
the Gurarij space \cite{gurarij1966spaces} in the category of Banach spaces;
in particular, the metric space underlying $\mathbb{G}_{r}(N)$ is isometric
to the Urysohn space of diameter $r$ if and only if $N\in \{0,2\}$ \cite[%
Theorems 5.1 and 5.5]{Niemiec}. Niemiec's construction builds on and
generalizes the construction of the group $\mathbb{G}_{\infty }(0)$ obtained
by Shkarin in \cite{shkarin1999universal}. The goal of this paper is to prove various properties of $\G_r(N)$, in particular we obtain the following results:
\begin{itemize}
    \item Let $\Gamma_0=\bigoplus_{n\in\mathbb N}\mathbb Q/\mathbb Z$ and $\Gamma_N=\bigoplus _{n\in\mathbb N}\mathbb Z/N\mathbb Z$ for $N\geq 2$. For all $r\in\{1,\infty\}$ and all $N\in\{0,2,3,\ldots\}$ there is a comeager set of values $\lambda$ on $\Gamma_N$, bounded by $1$ if $r=1$, such that the completion of $(\Gamma_N,\lambda)$ is isometrically group-isomorphic to $\G_r(N)$, see Theorem \ref{thm: generic metric completion is isomorphic to Niemiec's universal group}. This generalizes a result of Doucha, who established the case $N=0$, $r=\infty$ and proved that there exists a generic metric completion of $\Gamma_N$ for all $N$ and $r=\infty$ \cite[Theorem 0.2]{doucha2019}, without identifying it explicitly with $\G_r(N)$, although this identification is mentioned as expected \cite[Remark 3.16]{doucha2019}. To identify the generic metric completion of $\Gamma_N$ with $\G_r(N)$ we exploit a result of Niemiec giving a sufficient condition for a valued Abelian group to have $\mathbb{G}_r(N)$ as completion, in terms of an approximate extension property, analogous to the one used by Kubi\'{s} and Solecki \cite{kubis2013proof} in their proof of the uniqueness of the Gurarij space (see Proposition \ref{prop: characterizing Niemiec groups in terms of approximate extension property} below for a precise statement). 
    \item For every $r\in\{1,\infty\}$ and every $N\in\{0,2,3,\ldots\}$, the group $\G_r(N)$ is a \emph{L\'{e}vy group}. We verify this property directly in Theorem \ref{thm: Niemiec groups are Levy}, by building an explicit increasing  family of finite subgroups whose normalized Haar measures concentrate, using Niemiec's extension theorem \cite[Theorem 1.2(a)]{Niemiec} (restated as Theorem \ref{thm: extension property for G_r(N)} below) and an inequality by Talagrand on the concentration function of Hamming cubes (Example \ref{example: discrete spaces form a  Levy family} below). The argument is similar to the one used by Pestov to prove that $\mathrm{Iso}(\mathbb U)$, the isometry group of the Urysohn space, is L\'{e}vy by approximating it with finite subgroups \cite{pestov2005isometry}.
    \item For every $r\in\{1,\infty\}$ and every $N\in\{0,2,3,\ldots\}$ the group $\G_r(N)$ is \emph{strongly exotic}. We verify this property in Theorem \ref{theorem: G_r(N) is strongly exotic} by showing that every finite-order element of  $\G_r(N)$ is contained in a copy of the strongly exotic group $L^0(\phi,\mathbb Z/n\mathbb Z)$ \cite{schneider2025groups}.
    \item For every $r\in\{1,\infty\}$ and every $N\in\{0,2,3,\ldots\}$ the group $\G_r(N)$ is \emph{extremely amenable}. The extreme amenability of $\G_\infty(0)$ was already established by Doucha \cite[Corollary 3.17]{doucha2019} by using his result about the generic completion of $\Gamma_0$ mentioned in the first bullet point, combined with a criterion by Melleray and Tsankov concerning the extreme amenability of the generic completion of a countable group \cite[Theorem 6.4]{melleray2013generic}. Since Doucha's strategy can be carried out for all $r$ and $N$, once the generic metric completion of $\Gamma_N$ has been identified with $\G_r(N)$ and since being L\'{e}vy, as well as being strongly exotic and amenable, imply being extremely amenable, we obtain three distinct proofs of the extreme amenability of $\G_r(N)$.
    \item The groups $\G_1(0)$ and $\G_\infty(0)$ are \emph{generically monothetic}. This answers Question 1 of \cite{Niemiec} and shows that the Urysohn sphere $\mathbb U_1$ (which is isometric to $\G_1(0)$ by \cite[Theorem 5.1]{Niemiec}) carries a monothetic group structure, with respect to which the metric of $\mathbb U_1$ is invariant, a result established for the Urysohn space $\mathbb U$ by Cameron and Vershik \cite{cameron2006isometry}.
\end{itemize}
The paper is structured as follows. In Section \ref{section: preliminaries} we recall
the necessary background on valued groups, the properties of $\mathbb{G}_r(N)
$ which will be needed in the following arguments and the criterion of Melleray and Tsankov for establishing the extreme amenability of the generic completion of a countable group. In
Section \ref{section: extreme amenability} we prove that the generic
completion of $\Gamma_N$ is isomorphic to $\mathbb{G}_r(N)$ (Theorem \ref%
{thm: generic metric completion is isomorphic to Niemiec's universal group})
and conclude that $\mathbb{G}_r(N)$ is extremely amenable (Theorem \ref{thm:
extreme amenability of G_r(N)}). In Section \ref{section: G_r(N) is strongly exotic and Levy} we first prove that $\G_r(N)$ is a L\'{e}vy group (Theorem \ref{thm: Niemiec groups are Levy}) and recover the extreme amenability of $\G_r(N)$ (Corollary \ref{cor: extreme amenable through Levy}). We then prove that $\G_r(N)$ is strongly exotic (Theorem \ref{theorem: G_r(N) is strongly exotic}) and obtain as a corollary a third proof of the extreme amenability of $\G_r(N)$ (Corollary \ref{cor: extremely amenable through strongly exotic}). Finally, in Section \ref{section: G_r(0) is monothetic}, we show that $\G_r(0)$ is monothetic (Theorem \ref{thm: G_r(0) is generically monothetic}) and obtain the existence of a monothetic group structure on $\mathbb U_1$ in Corollary \ref{corollary: group structure on Urysohn sphere}.

\section{Preliminaries}

\label{section: preliminaries} In this section we recall the necessary
background on valued groups and the main properties of $\mathbb{G}_r(N)$
that will be needed in the following sections. We write
all groups in additive notation, in particular we denote the identity by $0$%
. A \emph{value} on a group $G$ is a function $p\colon G\to[0,\infty)$ such
that 
\begin{align*}
p(x)=0 &\iff x=0 \\
p(-x)&=p(x) \\
p(x+y)&\leq p(x)+p(y)
\end{align*}
for every $x,y\in G$. A value $p$ on a group induces a metric $%
d_p(x,y)=p(x-y)$, and we will often identify a value with the associated
metric. We denote by $\widehat{(G,p)}$ the completion of $G$ with respect to 
$p$.

A valued Abelian group $(G,p)$ is of \emph{class} $\mathcal{O}_0$ if 
\begin{equation*}
\lim_{n\to\infty}\frac{p(nx)}{n}=0
\end{equation*}
for every $x\in G$. Let $\mathfrak{G}$ denote the class of all separable
valued Abelian groups. Let $\mathfrak{G}_\infty(0)$ and $\mathfrak{G}_1(0)$
denote, respectively, the class of all separable valued Abelian groups $(G,p)
$ of class $\mathcal{O}_0$ and the subclass of those for which additionally $%
p\leq 1$. In particular, $\mathfrak{G}_1(0)\subseteq\mathfrak{G}_\infty(0)$.
Additionally, for $N>1$, let $\mathfrak{G}_\infty(N)$ denote the subclass of 
$\mathfrak{G}_\infty(0)$ consisting of groups of exponent dividing $N$, and let $%
\mathfrak{G}_1(N)=\mathfrak{G}_\infty(N)\cap\mathfrak{G}_1(0)$.

\begin{theorem}[{{\protect\cite[Theorem 1.1]{Niemiec}}}]
\label{thm: Niemiec characterization}  Let $r\in\{1,\infty\}$ and let $%
N\in\{0,2,3,4,\ldots\}$. There is a unique (up to isometric group
isomorphism) valued Abelian group, denoted by $\mathbb{G}_r(N)$, with the
following four properties: 

\begin{enumerate}
\item $\mathbb{G}_r(N)$ is complete and $\mathbb{G}_r(N)\in\mathfrak{G}_r(N)$%
. 

\item Every finite valued group in $\mathfrak{G}_r(N)$ admits an isometric
group embedding into $\mathbb{G}_r(N)$. 

\item Every isometric group homomorphism between two finite subgroups of $%
\mathbb{G}_r(N)$ is extendable to an isometric group homomorphism of $%
\mathbb{G}_r(N)$ onto itself. 

\item If $N=0$, elements of finite order form a dense subgroup of $\mathbb{G}%
_r(N)$. 
\end{enumerate}
\end{theorem}

\begin{remark}
    By \cite[Theorem 5.8]{Niemiec} the groups $\G_r(N)$ are pairwise non-isomorphic, as topological groups, for different choices of $r\in\{1,\infty\}$ and $N\in\{0,2,3,\ldots\}$. In particular $\G_1(N)$ and $\G_\infty(N)$ are distinct as topological groups, not just as metric groups. We will often prove properties of $\G_1(N)$ and $\G_\infty(N)$ in parallel by writing the argument for  $r=\infty$ and specifying where values need to be truncated to obtain the $r=1$ case as well.
\end{remark}

The extension property of $\G_r(N)$ in the next theorem will be essential in some of the following arguments.  

\begin{theorem}[\protect{\cite[Theorem 1.2(a)]{Niemiec}}]\label{thm: extension property for G_r(N)}
    Let $(H,q)\in\GG_r(N)$, let $K\leq H$ be a closed subgroup and let $\varphi\colon K\to\G_r(N)$ be a continuous group homomorphism whose range has compact closure in $\G_r(N)$. There is a continuous group homomorphism $\Phi\colon H\to\G_r(N)$ extending $\varphi$ with $\ker(\Phi)=\ker(\varphi)$. If $\varphi$ is a topological (resp. isometric) embedding, $\Phi$ can be constructed to be a topological (resp. isometric) embedding.
\end{theorem}
We will also need the following criterion, due to Niemiec, ensuring that a valued
groups $(G,p)$ satisfies that $\widehat{(G,p)}$ is isomorphic to $\mathbb{G}_r(N)$,
in terms of an approximate extension property. For $0<\varepsilon<1$ we say
that a homomorphism $u\colon (G,p)\to(F,q)$ between valued groups is \emph{$%
\varepsilon$-almost isometric} if 
\begin{equation*}
(1-\varepsilon)p(x)\leq q(u(x))\leq(1+\varepsilon)p(x),
\end{equation*}
for every $x\in G$. We say that $Q\subseteq[0,\infty)$ is \emph{admissible}
if $Q$ is countable, $0\in Q$, $Q$ is dense in $[0,\infty)$, and $%
Q+Q\subseteq Q$. Given an admissible set $Q$, a valued Abelian group $(G,p)$
is called \emph{$Q$-valued} if $p(G)\subseteq Q$.

\begin{proposition}[{{\protect\cite[Proposition 3.10(C)]{Niemiec}}}]
\label{prop: characterizing Niemiec groups in terms of approximate extension
property} Let $Q$ be an admissible set and let $(G,p)$ be a countable $Q$%
-valued Abelian group in $\mathfrak{G}_r(N)$ satisfying the following two
conditions:

\begin{enumerate}
\item whenever $(H,q)\in\mathfrak{G}_r(N)$ is a finite $Q$-valued group, $%
K\leq H$ is a subgroup and $\varphi\colon K\to G$ is an isometric group
embedding, then, for every $0<\varepsilon<1$, there exists an $\varepsilon$%
-almost isometric group embedding $\psi\colon (H,q)\to (G,p)$ satisfying 
\begin{equation*}
\|\psi_{|_K}-\varphi\|_\infty\leq\varepsilon,
\end{equation*}

\item the set of elements of finite order is dense in $G$.
\end{enumerate}

Then $\widehat{(G,p)}$ is isometrically group-isomorphic to $\mathbb{G}_r(N)$%
.
\end{proposition}

Given a countable Abelian group $G$, we denote by $\mathfrak{M}_\infty(G)$
the space of all values on $G$ and by $\mathfrak{M}_1(G)$ the subspace of
values bounded by $1$. We consider $\mathfrak{M}_r(G)$ as a subspace of $%
[0,1]^G$ or $[0,\infty)^G$, depending on whether $r=1$ or $r=\infty$, and we
equip it with the subspace topology induced by the product topology. Recall
that a \emph{semivalue} $p$ on a group $G$ is obtained by replacing the
first condition in the definition of a value by $p(0)=0$. In particular, the
space of all semivalues on $G$ is closed in $[0,\infty)^G$, since
it is defined by closed conditions. The space of values $\mathfrak{M}%
_\infty(G)$ is $G_\delta$ in the space of semivalues, since it is defined by
the intersection of the countably many open conditions $p(g)>0$ as $g$
ranges over the nonzero elements of $G$. It follows that both $\mathfrak{M}%
_\infty(G)$ and its closed subspace $\mathfrak{M}_1(G)$ are Polish. Whenever
we say that the generic value on a countable Abelian group $G$ satisfies
some property, we mean that the set of values in $\mathfrak{M}_1(G)$ or $%
\mathfrak{M}_\infty(G)$, depending on the context, that satisfy this
property is comeager.

As mentioned in the introduction we will use a result of Melleray and
Tsankov concerning the extreme amenability of the generic completion of a
countable group, but we need a few more definitions before being able to
state it. Recall that an Abelian group $G$ is said to have \emph{bounded
exponent} if there is some $n\in\mathbb{N}$ such that $nx=0$ for every $x\in
G$, and that $G$ is called \emph{unbounded} otherwise. By a classical result
of Pr\"{u}fer, if $G$ is an Abelian group with bounded exponent, then $G$
can be written as a direct sum of cyclic groups, called its \emph{Pr\"{u}fer
decomposition}. When $G$ is an Abelian group with bounded exponent we denote
by $m_G(p^n)$ the multiplicity of $\mathbb Z/p^n\mathbb Z$ in the Pr\"{u}fer
decomposition of $G$. The following condition was considered by Melleray and
Tsankov \cite{melleray2013generic}

\begin{equation}
m_G(p^n)>0\implies \exists k\geq n\,\left(m_G(p^k)=\infty\right)
\label{eq: exponent condition}
\end{equation}

and used to provide a criterion for the extreme amenability of the generic
metric completion of a countable group:

\begin{theorem}[{{\protect\cite[Theorem 6.4]{melleray2013generic}}}]
\label{thm: criterion for extreme amenability}  Let $G$ be a countable
Abelian group. If either $G$ is unbounded or $G$ is bounded and satisfies \eqref{eq:
exponent condition}, then the set of invariant metrics on $G$ whose
completion is extremely amenable is a dense $G_\delta$ set in the Polish
space of all invariant metrics on $G$.
\end{theorem}

In \cite{melleray2013generic}, Theorem 6.4 is stated in terms of the extreme
amenability of $(G,d)$ itself rather than of its completion; the two
formulations are equivalent, since a topological group is extremely amenable
if and only if its completion is (see the remarks in \cite[Section 3]%
{melleray2013generic}).

For every admissible set $Q$, Niemiec constructs a countable $Q$-valued group
$Q\mathbb{G}_r(N)$ as the Fra\"{\i}ss\'{e} limit of the finite $Q$-valued
groups in $\mathfrak{G}_r(N)$, and the completion of $Q\mathbb{G}_r(N)$ is
isometrically group-isomorphic to $\mathbb{G}_r(N)$ \cite[Theorem 3.2 and
Proposition 3.10(A)]{Niemiec}. See \cite{fraisse1954classifications} and
\cite[Chapter 7]{hodges1993model} for background on Fra\"{\i}ss\'{e} theory.
In particular, the following amalgamation lemma will be important:

\begin{lemma}[{{\protect\cite[Lemma 2.19]{Niemiec}}}]
\label{lemma: amalgamation}  Let $(G_i,p_i)\in\mathfrak{G}_r(N)$ for $%
i\in\{0,1,2\}$ and, for $j\in\{1,2\}$, let $\varphi_j\colon G_0\to G_j$ be
isometric embeddings. Then there are a valued group $(G,p)\in\mathfrak{G}%
_r(N)$ and isometric embeddings $\psi_j\colon G_j\to G$, for $j\in\{1,2\}$,
such that $\psi_1\circ\varphi_1=\psi_2\circ\varphi_2$. If, in addition, $%
Q\subseteq[0,\infty)$ is a set containing $0$, dense in $[0,\infty)$, and
satisfying $Q+Q\subseteq Q$ (countability of $Q$ is not required here), and $%
G_1$ and $G_2$ are finite and $Q$-valued, then $G$ is also finite and $Q$%
-valued. In particular, taking $Q=[0,\infty)$, if $G_1$ and $G_2$ are
finite, then $G$ can be taken to be finite.
\end{lemma}

\section{\texorpdfstring{The generic completion of $\Gamma_N$}{The generic completion of GammaN}}

\label{section: extreme amenability} Let $\mathfrak{F}_N$ be the class of
finite Abelian groups of exponent dividing $N$, for $N\geq 2$, and let $%
\mathfrak{F}_0$ be the class of all finite Abelian groups. By Lemma \ref%
{lemma: amalgamation}, $\mathfrak{F}_N$ is a Fra\"{\i}ss\'{e} class (we can
amalgamate discrete groups by considering them as valued groups, applying
the Lemma and forgetting the values), and we denote by $\Gamma_N$ its Fra%
\"{\i}ss\'{e} limit, which is a countable, locally finite group with age $%
\mathfrak{F}_N$ and the following extension property: whenever $B\in%
\mathfrak{F}_N$, $A\leq B$ and $f\colon A\to \Gamma_N$ is a group embedding,
there exists a group embedding $g\colon B\to\Gamma_N$ extending $f$. More
concretely, we have
\begin{equation*}
\Gamma_0\cong\bigoplus_{n\in\mathbb{N}}\mathbb{Q}/\mathbb{Z}\quad\text{and}\quad
\Gamma_N\cong\bigoplus_{n\in\mathbb{N}}\mathbb{Z}/N\mathbb{Z},
\qquad\text{for }N\geq 2.
\end{equation*}

First we verify that $\Gamma_N$ satisfies the hypothesis of Theorem \ref%
{thm: criterion for extreme amenability}. Afterward we will prove that the
generic metric completion of $\Gamma_N$ is isometrically group-isomorphic to
either $\mathbb{G}_\infty(N)$ or $\mathbb{G}_1(N)$ depending on whether we
work in the Polish space of all metrics on $\Gamma_N$ or those bounded by $1$%
.

\begin{lemma}
\label{lemma: hypothesis of criterion for extreme amenability are satisfied}
The group $\Gamma_0$ is unbounded. The group $\Gamma_N$ is bounded and
satisfies $\eqref{eq: exponent condition}$ for every $N\geq 2$.
\end{lemma}

\begin{proof}
Since $\Gamma_0$ is universal for finite Abelian groups, it contains
elements of arbitrarily high order and is unbounded.

For $N\geq 2$ it follows from the explicit description of $\Gamma_N$ above
that, for any prime $p$ and $n\in\mathbb{N}$, if $m_{\Gamma_N}(p^n)>0$, then
we already have $m_{\Gamma_N}(p^n)=\infty$, so that condition \eqref{eq:
exponent condition} is satisfied.  
\end{proof}

It remains to establish that the generic metric completion of $\Gamma_N$ is $%
\mathbb{G}_r(N)$. Fix $r\in\{1,\infty\}$ and $N\in\{0,2,3,\ldots\}$. There
is a difficulty in using Proposition \ref{prop: characterizing Niemiec
groups in terms of approximate extension property} in that we do not know
the admissible set $Q$ beforehand. We work around this issue by introducing
some codes (the expressions $P_h$ below) that represent $Q$ independently of
a value $\lambda$ and allow us to encode the triples $(H,K,\varphi)$ such
that $K\leq H$ and $\varphi\colon K\to\Gamma_N$ is an isometric embedding
before choosing a value $\lambda$ on $\Gamma_N$.

\begin{theorem}
\label{thm: generic metric completion is isomorphic to Niemiec's universal
group}  For every $r\in\{1,\infty\}$ and every $N\in\{0,2,3,\ldots\}$, the
set 
\begin{equation*}
\mathcal{F}_{r,N}=\{\lambda\in\mathfrak{M}_r(\Gamma_N)\mid \widehat{%
(\Gamma_N,\lambda)} \text{ is isometrically group-isomorphic to }\mathbb{G}%
_r(N)\}
\end{equation*}
is dense and comeager in $\mathfrak{M}_r(\Gamma_N)$.
\end{theorem}

\begin{proof}
We want to encode the approximate extension property of Proposition \ref%
{prop: characterizing Niemiec groups in terms of approximate extension
property} through countably many open dense conditions.

Let $\mathcal{P}$ be the collection of formal expressions 
\begin{equation*}
P=a+X_{x_1}+\ldots+X_{x_m}, \quad a\in\mathbb{Q}_{\geq 0},\, m\in\mathbb
N,\, x_i\in\Gamma_N,
\end{equation*}
where $m=0$ is possible. For a value $\lambda\in\mathfrak{M}_r(\Gamma_N)$,
we define the evaluation 
\begin{equation*}
\lambda(P)=a+\lambda(x_1)+\ldots+\lambda(x_m),
\end{equation*}
and we call $P$ syntactically positive if either $a>0$ or one of the $x_i$
is nonzero. In particular, $\lambda(P)>0$ for every $\lambda\in\mathfrak{M}%
_r(\Gamma_N)$ whenever $P$ is syntactically positive.

We fix the following finite data, corresponding to one instance of the
extension problem required by Proposition \ref{prop: characterizing Niemiec
groups in terms of approximate extension property}. We will show that the
set of metrics satisfying this particular instance is open and dense, and we
will then intersect over the countably many instances required to conclude
the proof.

\begin{enumerate}
\item A finite group $H\in\mathfrak{F}_N$ and $K\leq H$, 

\item a group embedding $\varphi\colon K\to\Gamma_N$, 

\item formal expressions $P_h\in\mathcal{P}$, for $h\in H$, with $P_0=0$ and 
$P_h$ syntactically positive for $h\neq 0$, 

\item a rational $\varepsilon\in(0,1)$. 
\end{enumerate}

For $\lambda\in\mathfrak{M}_r(\Gamma_N)$, define $q_\lambda(h)=\lambda(P_h)$%
, and let $A=A(H,K,\varphi,(P_h)_{h\in H})$ be the set of those $\lambda\in%
\mathfrak{M}_r(\Gamma_N)$ for which $q_\lambda$ is a value on $H$, bounded
by $1$ if $r=1$, and $\varphi\colon (K,q_\lambda)\to(\Gamma_N,\lambda)$ is
isometric. Note that $A$ is closed in $\mathfrak{M}_r(\Gamma_N)$. Indeed, $%
q_\lambda(h)>0$ for every $h\in H\setminus\{0\}$ is guaranteed by our choice
of $P_h$; the conditions $q_\lambda(0)=0$, $q_\lambda(h+h^{\prime })\leq
q_\lambda(h)+q_\lambda(h^{\prime })$, and $q_\lambda(h)=q_\lambda(-h)$ are
closed; and so is the condition $q_\lambda\leq 1$ when $r=1$. Finally, the
condition that $\varphi$ is an isometry is expressed by finitely many
equalities and is therefore closed. Consider now the set $%
B=B(H,K,\varphi,(P_h)_{h\in H},\varepsilon)$ consisting of all those $%
\lambda\in\mathfrak{M}_r(\Gamma_N)$ for which there exists a group embedding 
$\psi\colon H\to\Gamma_N$ such that 
\begin{align}
(1-\varepsilon)q_\lambda(h)&<\lambda(\psi(h))<(1+\varepsilon)q_\lambda(h),%
\quad & &\text{for }h\in H\setminus\{0\},  \label{condition: 2} \\
\lambda(\psi(k)-\varphi(k))&<\varepsilon\quad & &\text{for }k\in K.
\label{condition: 3}
\end{align}
Note that for every group embedding $H\to\Gamma_N$ conditions %
\eqref{condition: 2} and \eqref{condition: 3} are open, so that $B$ is an
open set. It follows that the set 
\begin{equation*}
U=U(H,K,\varphi,(P_h)_{h\in H},\varepsilon)=(\mathfrak{M}_r(\Gamma_N)%
\setminus A)\cup B
\end{equation*}
is open. Intuitively membership in $U$ expresses the implication \enquote{$%
\lambda\in A\implies\lambda\in B$}: either the fixed codes $(P_h)_{h\in H}$
do not express an extension problem from $K$ to $H$ for $\lambda$, so that
the condition is vacuously satisfied, or they do define an extension problem
and $\lambda\in B$ ensures it can be solved within $\varepsilon$. We claim
that $U$ is dense, but before checking its density let us assume it and
conclude the argument.

Assume that $U(H,K,\varphi,(P_h)_{h\in H},\varepsilon)$ is open dense for
every choice of finite data, and let $\mathcal{G}_{r,N}$ be the intersection
of all these sets, ranging over finite $H\in\mathfrak{F}_N$, subgroups $%
K\leq H$, embeddings $\varphi\colon K\to\Gamma_N$, choices of codes $%
(P_h)_{h\in H}$, and rational $\varepsilon\in(0,1)$. Since there are only
countably many embeddings $K\to\Gamma_N$ and the other data also range over
countable sets, $\mathcal{G}_{r,N}$ is a dense $G_\delta$ set. We show that
every $\lambda\in\mathcal{G}_{r,N}$ satisfies the hypotheses of Proposition %
\ref{prop: characterizing Niemiec groups in terms of approximate extension
property}, so that $\widehat{(\Gamma_N,\lambda)}\cong\mathbb{G}_r(N)$.
Define 
\begin{equation*}
Q_\lambda=\left\{a+\sum_{j=1}^m\lambda(x_j)\mid a\in\mathbb{Q}_{\geq 0},\
m\in\mathbb{N},\ x_j\in\Gamma_N \right\}.
\end{equation*}

The set $Q_\lambda$ is a countable admissible subset of $[0,\infty)$
containing $\lambda(\Gamma_N)$. Fix a $Q_\lambda$-valued finite group $(H,q)\in\GG_r(N)$, with $q\leq1$ if $r=1$, a subgroup $K\leq H$, an isometric embedding $%
\varphi\colon K\to\Gamma_N$, and $0<\varepsilon<1$. We need to produce an $%
\varepsilon$-almost isometric group embedding $\psi\colon(H,q)\to(\Gamma_N,%
\lambda)$ such that $\|\psi_{|_K}-\varphi\|_\infty\leq\varepsilon$. For each 
$h\in H\setminus\{0\}$, choose a code $P_h$ with $\lambda(P_h)=q(h)$; this
code is necessarily syntactically positive since $q(h)>0$. Let $P_0=0$. By
our choice of codes, $\lambda\in A(H,K,\varphi,(P_h)_{h\in H})$. Choose a
rational $0<\varepsilon^{\prime }<\varepsilon$. Since $\lambda\in\mathcal{G}%
_{r,N}\subseteq U(H,K,\varphi,(P_h)_{h\in H},\varepsilon^{\prime })$, it
follows that $\lambda\in B(H,K,\varphi,(P_h)_{h\in H},\varepsilon^{\prime })$%
. Thus there is an $\varepsilon^{\prime }$-almost isometric embedding $\psi$
with $\|\psi_{|_K}-\varphi\|_\infty\leq\varepsilon^{\prime }$, which gives
the approximate extension property required by Proposition \ref{prop:
characterizing Niemiec groups in terms of approximate extension property}.
Since $\Gamma_N$ is locally finite and hence torsion, the second condition
of that proposition is also satisfied. Moreover $(\Gamma_N,\lambda)$ is of class $\mathcal O_0$, being torsion, and we have $(\Gamma_N,\lambda)\in\GG_r(N)$ since $\Gamma_N$ has exponent dividing $N$ and value bounded, if $r=1$, by construction. Therefore $\widehat{(\Gamma_N,\lambda)%
}\cong\mathbb{G}_r(N)$. In particular, $\mathcal{G}_{r,N}\subseteq\mathcal{F}%
_{r,N}$, showing that the latter contains a comeager set in $\mathfrak{M}%
_r(\Gamma_N)$.

It remains to show that, for every choice of finite data, the corresponding
set $$U=U(H,K,\varphi,(P_h)_{h\in H},\varepsilon)$$ is dense. Let $\lambda_0\in%
\mathfrak{M}_r(\Gamma_N)$, and let $V$ be a basic open neighbourhood of $%
\lambda_0$. Thus, for some finite $F\subseteq\Gamma_N$ and some $\delta>0$,
we may write 
\begin{equation*}
V=\left\{\lambda\in\mathfrak{M}_r(\Gamma_N)\mid
\lvert\lambda(f)-\lambda_0(f)\rvert<\delta\text{ for every }f\in F\right\}.
\end{equation*}
If $V\cap(\mathfrak{M}_r(\Gamma_N)\setminus A)\neq\varnothing$, then $V\cap
U\neq\varnothing$, so assume that $V\subseteq A$. Write $q_0=q_{\lambda_0}$.
Since $\lambda_0\in A$, the function $q_0$ is a value on $H$, bounded by $1$
if $r=1$, and $\varphi\colon(K,q_0)\to(\Gamma_N,\lambda_0)$ is an isometric
embedding.

Let $C$ be the subgroup of $\Gamma_N$ generated by $F$, $\varphi(K)$, and
all the elements of $\Gamma_N$ appearing in the expressions $(P_h)_{h\in H}$%
. This subgroup is finite because $\Gamma_N$ is locally finite. Apply Lemma %
\ref{lemma: amalgamation} to $(H,q_0)$ and $(C,\lambda_0)$ over $K$,
identifying $K$ with $\varphi(K)$. We obtain a finite valued group $(D,s)\in%
\mathfrak{G}_r(N)$ and isometric embeddings 
\begin{equation*}
i_H\colon(H,q_0)\to(D,s),\qquad i_C\colon(C,\lambda_0)\to(D,s)
\end{equation*}
such that $i_H\circ\iota_K=i_C\circ\varphi$, where $\iota_K\colon K\to H$ is
the inclusion. Since $D$ is finite and, when $N\geq2$, of exponent dividing $%
N$ because $(D,s)\in\mathfrak{G}_r(N)$, we have $D\in\mathfrak{F}_N$. By the
extension property of the Fra\"{\i}ss\'{e} limit $\Gamma_N$, the embedding $%
i_C(c)\mapsto c$ of $i_C(C)$ into $\Gamma_N$ extends to an embedding $%
j\colon D\to\Gamma_N$. Hence $j\circ i_C=\mathrm{id}_C$.

Define a value $s_j$ on $j(D)$ by $s_j(j(d))=s(d)$. It extends $\lambda_0|_C$%
. If $r=1$, set $M=1$; otherwise, choose 
\begin{equation*}
M\geq\max\{1,s(d)\mid d\in D\}.
\end{equation*}
Define $\lambda\colon\Gamma_N\to[0,\infty)$ by 
\begin{equation*}
\lambda(x)=%
\begin{cases}
s_j(x), & x\in j(D), \\ 
M, & x\notin j(D).%
\end{cases}%
\end{equation*}
This is a value on $\Gamma_N$, bounded by $1$ when $r=1$. Indeed,
definiteness and symmetry are immediate. For subadditivity, the case $x,y\in
j(D)$ follows from the subadditivity of $s_j$. If exactly one of $x,y$ lies
outside $j(D)$, then $x+y\notin j(D)$ and the inequality is immediate. If
both lie outside $j(D)$, then either $x+y\notin j(D)$, in which case $M\leq2M
$, or $x+y\in j(D)$, in which case $s_j(x+y)\leq M\leq2M$. When $r=1$, the
inequality $s_j\leq1=M$ follows because $(D,s)\in\mathfrak{G}_1(N)$.

The value $\lambda$ agrees with $\lambda_0$ on $C$, so $\lambda\in V$. Set $%
\psi=j\circ i_H$. Then $\psi|_K=\varphi$. Moreover, since every element
appearing in each code $P_h$ belongs to $C$, 
\begin{equation*}
\lambda(P_h)=\lambda_0(P_h)=q_0(h)=s(i_H(h))=\lambda(\psi(h))
\end{equation*}
for every $h\in H$. Thus conditions \eqref{condition: 2} and %
\eqref{condition: 3} hold exactly, and $$\lambda\in B(H,K,\varphi,(P_h)_{h\in
H},\varepsilon)\cap V.$$ This proves that $U$ is dense.
\end{proof}

\begin{remark}
    Doucha proved \cite[Theorem 0.2]{doucha2019} that whenever $G$ is a countable Abelian group with $G\cong\bigoplus_{n\in\mathbb N}G$, there exists a generic metric completion of $G$. He asked whether there exists a single generic Abelian group, that is a single group $\mathbb H$ such that whenever $G_1$ and $G_2$ are countable Abelian groups isomorphic to their countable direct sum, the generic completions of both $G_1$ and $G_2$ are isomorphic to $\mathbb H$. He observed that if $G_1$ and $G_2$ are torsion groups with bounded torsion but different bounds on the torsion, then they cannot have the same generic completion, but stated as a reasonable expectation the fact that the groups $\G_\infty(N)$ are the only generic Abelian groups \cite[Remark 3.16]{doucha2019}. Theorem \ref{thm: generic metric completion is isomorphic to Niemiec's universal group} confirms this expectation for $G=\Gamma_N$, though the general statement remains open.
\end{remark}

Before obtaining the extreme amenability of $\G_1(N)$ for all $N$, we need a version of Theorem %
\ref{thm: criterion for extreme amenability} for $\mathfrak{M}_1(G)$ instead
of $\mathfrak{M}_\infty(G)$. This is an immediate consequence of the
unbounded result of Melleray and Tsankov, but we state it explicitly for
completeness.

\begin{theorem}
\label{thm: generic bounded completion is extremely amenable}  Let $G$ be a
countable Abelian group. If either $G$ is unbounded or $G$ is bounded and satisfies %
\eqref{eq: exponent condition}, the set  
\begin{equation*}
\mathfrak{E}=\left\{\lambda\in\mathfrak{M}_1(G)\mid \widehat{(G,\lambda)}%
\text{ is extremely amenable}\right\},
\end{equation*}
is a dense $G_\delta$ subset of $\mathfrak{M}_1(G)$.
\end{theorem}

\begin{proof}
The set $\mathfrak{E}$ is $G_\delta$, since it is the intersection of the $%
G_\delta$ subset of $\mathfrak{M}_\infty(G)$ given by Theorem \ref{thm:
criterion for extreme amenability} with the closed subspace $\mathfrak{M}%
_1(G)$. To prove density, fix $\lambda_0\in\mathfrak{M}_1(G)$, a finite set $%
F\subseteq G$, and $\varepsilon>0$. By Theorem \ref{thm: criterion for
extreme amenability}, there is a value $\mu\in\mathfrak{M}_\infty(G)$ such
that $\lvert\mu(f)-\lambda_0(f)\rvert<\varepsilon$ for every $f\in F$ and $%
\widehat{(G,\mu)}$ is extremely amenable. Let $\mu^{\prime
}(g)=\min\{1,\mu(g)\}$ for every $g\in G$. This truncation is a value
bounded by $1$, so $\mu^{\prime }\in\mathfrak{M}_1(G)$. Moreover, since $%
\lambda_0\leq1$,  
\begin{equation*}
\lvert\mu^{\prime
}(f)-\lambda_0(f)\rvert\leq\lvert\mu(f)-\lambda_0(f)\rvert\quad\text{for
every }f\in F.
\end{equation*}
Since $\mu$ and $\mu^{\prime }$ induce the same uniformity on $G$, their
completions are isomorphic as topological groups. Hence $\widehat{%
(G,\mu^{\prime })}$ is extremely amenable, proving that $\mu^{\prime }\in%
\mathfrak{E}$ and that $\mathfrak{E}$ is dense.
\end{proof}

We obtain as a corollary a proof of the extreme amenability of $\G_r(N)$.

\begin{theorem}
\label{thm: extreme amenability of G_r(N)}  For every $r\in\{1,\infty\}$ and
every $N\in\{0,2,3,\ldots\}$, the group $\mathbb{G}_r(N)$ is extremely
amenable.
\end{theorem}

\begin{proof}
By Theorem \ref{thm: generic metric completion is isomorphic to Niemiec's
universal group}, there is a comeager set of values in $\mathfrak{M}%
_r(\Gamma_N)$ whose completions are isomorphic to $\mathbb{G}_r(N)$. By
Lemma \ref{lemma: hypothesis of criterion for extreme amenability are
satisfied}, the group $\Gamma_N$ satisfies the hypotheses of Theorem \ref%
{thm: criterion for extreme amenability} when $r=\infty$ and of Theorem \ref%
{thm: generic bounded completion is extremely amenable} when $r=1$. Thus
there is also a comeager set of values in $\mathfrak{M}_r(\Gamma_N)$ whose
completions are extremely amenable. These two comeager sets intersect, so $%
\mathbb{G}_r(N)$ is extremely amenable.
\end{proof}

\section{\texorpdfstring{$\G_r(N)$ is a strongly exotic L\'{e}vy group}{Gr(N) is a strongly exotic Levy group}}
\label{section: G_r(N) is strongly exotic and Levy}

In this section, we show that the groups $\G_r(N)$ are strongly exotic and L\'{e}vy. The notion of a L\'{e}vy sequence was first defined and linked to extreme amenability by Gromov and Milman \cite{gromov1983topological}; we follow the presentation from \cite{pestov2017amenability}.
Let $(X,d,\mu)$ be a metric space equipped with a Borel probability measure $\mu$. The \emph{concentration function} of $(X,d,\mu)$, $\alpha_{(X,d,\mu)}\colon[0,\infty)\to [0,1/2]$ is defined by $$\alpha_{(X,d,\mu)}(\varepsilon)=\begin{cases}1/2 & \varepsilon=0\\
1-\inf\{\mu(B_d(A,\varepsilon))\mid A\subseteq X\text{ Borel}, \mu(A)\geq 1/2\} & \text{otherwise,}
    
\end{cases}$$
where $B_d(A,\varepsilon)$ is the $\varepsilon$-neighbourhood of $A$ with respect to $d$. A sequence $(X_i,d_i,\mu_i)$ of metric spaces with Borel probability measures is said to be a \emph{L\'{e}vy sequence} if every sequence of Borel subsets $A_i\subseteq X_i$ with $\mu_i(A_i)\geq 1/2$ satisfies $\lim_{i\to\infty}\mu_i(B_{d_i}(A_i,\varepsilon))=1$ for every $\varepsilon>0$. Equivalently a sequence $(X_i,d_i,\mu_i)$ is a L\'{e}vy sequence if for every $\varepsilon>0$, $\lim_{i\to\infty}\alpha_{(X_i,d_i,\mu_i)}(\varepsilon)=0$ \cite{gromov1983topological}.

The following example of a L\'{e}vy sequence  will be important in the  proof of Theorem \ref{thm: Niemiec groups are Levy}.

\begin{example}[\protect{\cite[Example 3.3]{pestov2017amenability}}] \label{example: discrete spaces form a  Levy family}
    Let $(X,\mu)$ be a discrete probability space and consider the sequence $(X^n,\delta_n,\mu^{\otimes n})$, where $\mu^{\otimes n}$ is the product measure and $\delta_n$ is the normalized Hamming distance $$\delta_n((x_1,\ldots,x_n),(x_1',\ldots,x_n'))=\frac1n\left|\left\{ i\mid x_i\neq x_i'\right\}\right|.$$  Then, for every $\varepsilon>0$, we have $\alpha_{(X^n,\delta_n,\mu^{\otimes n})}(\varepsilon)\leq 2\exp(-\varepsilon^2n)$ \cite[Proposition 2.1.1]{talagrand1995concentration}, in particular the sequence $(X^n,\delta_n,\mu^{\otimes n})$ is a L\'{e}vy sequence.
\end{example}

A Polish group $(G,d)$ with a compatible left-invariant metric $d$ is called:
\begin{itemize}
    \item A \emph{L\'{e}vy group}, if it has an increasing sequence of compact subgroups $(K_i)_{i\in\mathbb N}\subseteq G$ such that $\bigcup_i K_i$ is dense in $G$ and $(K_i,d_{|K_i},\mu_i)$ is a L\'{e}vy sequence, where $\mu_i$ is the normalized Haar measure on $K_i$. Since all metric are uniformly equivalent on a compact space, this definition is independent of the choice of $d$.
    \item \emph{Whirly amenable} \cite{pestov2017amenability}, if it is amenable and whenever $G$ acts continuously on a compact Hausdorff space $X$, any invariant regular Borel probability measure $\mu$ on $X$ is supported on the set of fixed points for the action.
    \item \emph{Exotic}  \cite{herer1975existence}, if every strongly continuous unitary representation of $G$ on a Hilbert space is trivial, that is if every continuous group homomorphism $G\to U(H)$ to the unitary group of a Hilbert space, equipped with the strong operator topology, is constant.
    \item \emph{Strongly exotic} \cite{banaszczyk1983existence}, if every weakly continuous representation of $G$ on a Hilbert space is trivial, that is if every continuous group homomorphism $G\to \mathrm{GL}(H)$ to the group of bounded, invertible, linear operators on a Hilbert space, equipped with the weak operator topology, is constant.
\end{itemize}

Every L\'{e}vy group is whirly amenable \cite[Remark 1.4]{glasner2005automorphism} and every whirly amenable group is extremely amenable (if $G\curvearrowright X$ is a continuous action of a whirly amenable group on a nonempty compact space, there exists an invariant probability measure with nonempty support because $G$ is amenable, and it is supported on the set of fixed points since $G$ is whirly amenable. In particular the set of fixed points is nonempty and $G$ is extremely amenable). Moreover every strongly exotic group is exotic and every amenable exotic group is whirly amenable \cite[Remark 5.7]{schneider2025groups}. Note that the properties of being L\'{e}vy and of being strongly exotic are incomparable: $U(\ell^2)$, the unitary group of the separable Hilbert space, is L\'{e}vy \cite{gromov1983topological} but not strongly exotic, while for a pathological submeasure $\phi$, the group $L^0(\phi,\mathbb R)$ is strongly exotic \cite[Theorem 5.5]{schneider2025groups} but not L\'{e}vy, since it is a real topological vector space so it contains no nontrivial compact subgroups. In addition being L\'{e}vy is strictly stronger than being whirly amenable, by the properties of $L^0(\phi,\mathbb R)$ already mentioned, and being whirly amenable is strictly stronger than being extremely amenable: for example $\mathrm{Aut}(\mathbb Q,\leq)$, the group of order-preserving bijections of $\mathbb Q$, equipped with the pointwise convergence topology, is extremely amenable \cite{pestov1998free} but not whirly \cite[Remark 1.3]{glasner2005automorphism}.

\begin{theorem}\label{thm: Niemiec groups are Levy}
    For every $r\in\{1,\infty\}$ and every $N\in\{0,2,3,\ldots\}$, the group $\G_r(N)$ is a L\'{e}vy group.
\end{theorem}
\begin{proof}
    Let $p$ be the value on $\G_r(N)$ and let $(a_n)_{n\in\mathbb N}$ be a dense sequence of finite-order elements of $\G_r(N)$ (when $N=0$ there exists such a sequence thanks to condition (4) in Theorem \ref{thm: Niemiec characterization}). We want to inductively construct a sequence $(H_n)_n$ of compact (in fact they will be finite) subgroups of $\G_r(N)$ that, equipped with the normalized Haar measure, form a L\'{e}vy sequence. For the base case let $H_0=\{0\}$. Suppose now that $H_{n-1}$ has been defined and let $A_n=\langle H_{n-1},a_n\rangle$. Let $D_n=\max\{1,\diam_p(A_n)\}$ and choose a sufficiently big positive integer $m_n$ with $m_n/D_n^2\geq n$. Let $B_n=A_n^{m_n}$ and equip it with the averaged value $$q_n(b_1,\ldots,b_{m_n})=\frac1{m_n}\sum_{i=1}^{m_n}p(b_i).$$ Note that $B_n$ is of class $\mathcal O_0$, has exponent dividing $N$ if $N\geq 2$, and satisfies $q_n\leq 1$ if $r=1$, that is $B_n\in \GG_r(N)$. Moreover the diagonal embedding $\Delta\colon A_n\to B_n$, $\Delta(a)=(a,\ldots,a)$ is an isometry, so that, by Theorem \ref{thm: extension property for G_r(N)}, there is an isometric embedding $\psi_n\colon B_n\to\G_r(N)$ with $(\psi_n\circ\Delta)(a)=a$ for every $a\in A_n$. Set $H_n=\psi_n(B_n)$ and note that the sequence $(H_n)_{n\in\mathbb N}$ is increasing with dense union, since $a_n\in H_n$. It remains to prove that the sequence $(H_n,p_{|H_n},\mu_n)$, where $\mu_n$ is the normalized Haar measure on $H_n$, is a L\'{e}vy sequence.  Note that the normalized Haar measure $\mu_n$ on $H_n$ is simply the normalized counting measure $\mu_n(A)=|A|/|H_n|$, for $A\subseteq H_n$, since $H_n$ is finite. Moreover, if $\nu_n$ denotes the normalized counting measure on $A_n$, we have that, since $\psi_n$ is a bijection from $B_n$ to $H_n$, $(\psi_n)_\ast(\nu_n^{\otimes m_n})=\mu_n$. 
    
    Let $\delta_n$ be the normalized Hamming distance on $B_n$, that is $$\delta_n((x_1,\ldots,x_{m_n}),(y_1,\ldots,y_{m_n}))=\frac{|\{i\leq m_n\mid x_i\neq y_i\}|}{m_n}.$$ If $x=(x_i)_i,y=(y_i)_i\in B_n$ we have $x_i-y_i\in A_n$ for every $i$, so that $p(x_i-y_i)\leq\diam_p(A_n)\leq D_n$, and we obtain $$p(\psi_n(x)-\psi_n(y))=q_n(x-y)=\frac1{m_n}\sum_{i=1}^{m_n}p(x_i-y_i)\leq \frac{D_n}{m_n}\left|\{j\mid x_j\neq y_j\}\right|\leq D_n\delta_n(x,y).$$

    Let $E_n\subseteq H_n$ with $\mu_n(E_n)\geq 1/2$ and let $\widetilde{E}_n=\psi_n^{-1}(E_n)$. By the previous estimate we have $$B_{\delta_n}(\widetilde{E}_n,\varepsilon/D_n)\subseteq\psi_n^{-1}(B_p(E_n,\varepsilon)),$$ for any $\varepsilon>0$. Since $\nu_n^{\otimes m_n}(\widetilde{E}_n)=\mu_n(E_n)$, it follows that 
    $$
        1-\mu_n(B_p(E_n,\varepsilon))\leq 1-\nu_n^{\otimes m_n}(B_{\delta_n}(\widetilde{E}_n,\varepsilon/D_n))\leq\alpha_{(B_n,\delta_n,\nu_n^{\otimes m_n})}(\varepsilon/D_n).
    $$

    Taking the supremum over all the sets $E_n\subseteq H_n$ with $\mu_n(E_n)\geq 1/2$ and applying the estimate from Example \ref{example: discrete spaces form a  Levy family} we obtain $$\alpha_{(H_n,p,\mu_n)}(\varepsilon)\leq\alpha_{(B_n,\delta_n,\nu_n^{\otimes m_n})}(\varepsilon/D_n)\leq 2\exp\left(-\frac{\varepsilon^2m_n}{D_n^2}\right)\leq 2\exp(-\varepsilon^2n),$$ where the last inequality also uses that $m_n/D_n^2\geq n$ by our choice of $m_n$. Since $2\exp(-\varepsilon^2n)\to 0$ as $n\to\infty$ we have that $(H_n,p_{|H_n},\mu_n)$ is a L\'{e}vy sequence. In particular $\G_r(N)$ is a L\'{e}vy group, as desired.
\end{proof}

As a corollary, we have a second proof of the extreme amenability of $\G_r(N)$.

\begin{corollary}\label{cor: extreme amenable through Levy}
    Since $\G_r(N)$ is L\'{e}vy, it is also whirly amenable hence extremely amenable. 
\end{corollary}

In order to prove that $\G_r(N)$ is strongly exotic we will leverage the existence of a strongly exotic group with some specific properties, which we recall in the following. The exact construction of this group is outside the scope of this paper and we refer the reader to \cite{schneider2025groups} for the details.

\begin{fact}\label{fact: properties of G_N}
    Let $\phi$ be a non-zero pathological submeasure on a countable Boolean algebra, which exists by \cite{herer1975existence} and \cite[Corollary A.9]{schneider2025groups} and consider, for $N\geq 2$, the group $G_N=L^0(\phi,\mathbb Z/N\mathbb Z)$ \cite[Definition 2.2]{schneider2025groups}. This group has the following properties:
    \begin{enumerate}
        \item The group $G_N$ is completely metrizable and separable. It is complete being constructed as a Raikov completion \cite[Definition 2.2]{schneider2025groups}, it is metrizable because $S(\phi,\mathbb Z/N\mathbb Z)$ is pseudometrizable \cite[Remark 2.6]{schneider2025groups} and $G_N$ is the completion of its Hausdorff quotient and separable because it is built from a countable Boolean algebra.
        \item The constant functions form a closed copy of $\mathbb Z/N\mathbb Z$ in $G_N$, there is such a copy by \cite[Remark 2.5]{schneider2025groups}, it is closed being compact.
        \item $G_N$ is strongly exotic \cite[Theorem 5.5]{schneider2025groups}.
        \item The group $G_N$ is of class $\mathcal O_0$, since $Ng=0$ for every $g\in G_N$.
    \end{enumerate}
\end{fact}

\begin{theorem}\label{theorem: G_r(N) is strongly exotic}
    For every $r\in\{1,\infty\}$ and every $N\in\{0,2,3,\ldots\}$, the group $\G_r(N)$ is strongly exotic.
\end{theorem}
\begin{proof}
    Let $\pi\colon \G_r(N)\to\mathrm{GL}(H)$ be a weakly continuous linear representation of $\G_r(N)$ on a Hilbert space. We want to show that it is trivial, which we will obtain by showing that $\pi(x)=\mathrm{id}_H$ for every $x\in \G_r(N)$ with finite order. Since the elements with finite order of $\G_r(N)$ are dense in $\G_r(N)$ if $N=0$ by condition $(4)$ of Theorem \ref{thm: Niemiec characterization} (and are the whole of $\G_r(N)$ if $N\geq 2$), this will show that $\pi$ is constant on $\G_r(N)$ by continuity. 

    Let $x\in\G_r(N)$ have finite order $n>1$ (clearly $\pi(x)=\mathrm{id}_H$ for $x=0)$ and consider the group $G_n$ from Fact \ref{fact: properties of G_N}, equipped with a compatible invariant value $q_n$ with $q_n\leq1$. If $N=0$ we already have that $G_n\in\GG_r(0)$ by Fact \ref{fact: properties of G_N}, while if $N\geq 2$ we must have $n\mid N$ so that $G_n$ has exponent dividing $N$ and we still obtain $G_n\in\GG_r(N)$. Since $G_n$ contains a closed copy of $\mathbb Z/n\mathbb Z$ we can define a  map $\theta_x\colon\mathbb Z/n\mathbb Z\to\G_r(N)$ by $\theta_x(k)=kx$ on this copy of $\mathbb Z/n\mathbb Z$. It is injective, since $x$ has order exactly $n$, and defined on a finite, hence closed, subgroup so, by Theorem \ref{thm: extension property for G_r(N)}, it can be extended to a continuous injective homomorphism $j_x\colon G_n\to\G_r(N)$. Since $G_n$ is strongly exotic we have that $\pi\circ j_x\colon G_n\to\mathrm{GL}(H)$ is trivial, so that in particular $\pi(x)=\mathrm{id}_H$.

    We have shown that $\pi(x)=\mathrm{id}_H$ for every $x\in\G_r(N)$ with finite order, which, by continuity of $\pi$ and the density of the finite-order elements in $\G_r(N)$, shows that $\pi(x)=\mathrm{id}_H$ for every $x\in\G_r(N)$.
    
\end{proof}

As a corollary, we have a third proof of the extreme amenability of $\G_r(N)$.

\begin{corollary}\label{cor: extremely amenable through strongly exotic}
    Since $\G_r(N)$ is strongly exotic, it is also exotic, and since it is also amenable (being an Abelian group) it is whirly amenable and extremely amenable.
\end{corollary}

\section{\texorpdfstring{$\G_r(0)$ is monothetic}{Gr(0) is monothetic}}
\label{section: G_r(0) is monothetic}
Recall that a topological group $G$ is called \emph{monothetic} if it contains a dense cyclic subgroup. For $N\geq 2$, $\G_r(N)$ is not monothetic, being an infinite torsion group, but we will show in this section that $\G_r(0)$ is monothetic. By the results of this section and the previous one, $\G_r(0)$ is a new example of a monothetic L\'{e}vy group; the first example of such a group was found by Glasner \cite{glasner1998minimal} and independently by Furstenberg and Weiss (unpublished). 

The argument establishing the monotheticity of $\G_r(0)$ will be a Baire category argument, showing the stronger statement that $\G_r(0)$ is \emph{generically monothetic}, meaning that $\{x\in \G_r(0)\mid\overline{\langle x\rangle}=\G_r(0)\}$ is comeager in $\G_r(0)$, but before carrying it  out we need to establish the existence of dyadic roots for the finite-order elements of $\G_r(0)$.

\begin{lemma}\label{lemma: exact halves in G_r(0)}
    Let $r\in\{1,\infty\}$ and let $G=\G_r(0)$ equipped with its value $p$. For every finite-order $a\in G$ there exists a finite-order $b\in G$ with $$2b=a\quad\text{ and }\quad p(b)=\frac12p(a).$$
\end{lemma}
\begin{proof}
    If $a=0$ simply take $b=0$, so suppose that $a$ is a nonzero, finite-order element of $G$. The goal is to embed $\langle a\rangle$ into a group of class $\GG_r(0)$ in which there exists a finite-order $b$ as in the statement of the lemma, and then use the extension property of $G$ (Theorem \ref{thm: extension property for G_r(N)}) to conclude. Let $K=\langle a\rangle$. If $\mathrm{ord}(a)=m$, we construct a cyclic group of order $2m$ by adjoining a formal half of $a$ to $K$. Let $H=\langle K,t\mid 2t=a\rangle$ and consider the value $q_0$ defined by $q_0(t)=p(a)/2$ and by the smallest possible value on the remaining elements. Concretely every element of $H$ can be written as $k+nt$ for some $k\in K$ and $n\in\mathbb Z$, but this representation is not unique, since, using the relation $2t=a$, we have $$k+nt=(k-ja)+(n+2j)t,$$ for any $j\in\mathbb Z$ and thus we define $$q_0(k+nt)=\inf_{j\in\mathbb Z}\left\{p(k-ja)+\frac{p(a)}2|(n+2j)|\right\}.$$
    If $r=\infty$ let $q=q_0$, otherwise, if $r=1$, let $q=\min\{1,q_0\}$. It is easy to check that $q$ is a value on $H$ and that the natural embedding $(K,p)\to (H,q)$ is isometric. Moreover, since $H$ is finite ($|H|=2|K|$), we have that $H\in\GG_r(0)$. By Theorem \ref{thm: extension property for G_r(N)} there is an isometric embedding $\psi\colon(H,q)\to (G,p)$ extending the inclusion of $(K,p)$ in $(G,p)$. The finite-order element $b=\psi(t)$ of $G$ is as in the statement of the lemma.
\end{proof}

\begin{corollary}\label{cor: dyadic roots}
    By iterating Lemma \ref{lemma: exact halves in G_r(0)} we obtain that, for every finite-order $a\in\G_r(0)$, there exists a $2^n$-th root of $a$, that is a finite-order $b\in\G_r(0)$ with $$2^nb=a\quad\text{ and }\quad p(b)=\frac1{2^n}p(a).$$
\end{corollary}

We can now verify that the generic element of $\G_r(0)$ generates a dense cyclic subgroup. This  gives a positive answer to Question 1 of \cite{Niemiec}, which asks whether the groups $\G_1(0)$ and $\G_\infty(0)$ are monothetic. For a topological group $G$, let $\Gen(G)=\{x\in G\mid\overline{\langle x\rangle}=G\}$ denote the set of \emph{topological generators} of $G$.

\begin{theorem}\label{thm: G_r(0) is generically monothetic}
    For $r\in\{1,\infty\}$, $\Gen(\G_r(0))$ is comeager in $\G_r(0)$, that is $\G_r(0)$ is generically monothetic.
\end{theorem}
\begin{proof}
    Let $G=\G_r(0)$ equipped with its value $p$. Fix a dense sequence $(a_i)_{i\in\mathbb N}$ of finite-order elements of $G$, which exists by condition $(4)$ in Theorem \ref{thm: Niemiec characterization}. For $j,k\in\mathbb N_{\geq 1}$ let $$\mathcal U_{j,k}=\bigcup_{m\in\mathbb Z}\left\{x\in G\mid p(mx-a_j)< 1/k\right\}.$$ 
    It is clear that each $\mathcal U_{j,k}$ is open, since $p$ and the map $x\mapsto mx$ for a fixed $m$ are continuous, moreover, since the sequence $(a_i)_i$ is dense in $G$, we also have $\Gen(G)=\bigcap_{j,k}\mathcal U_{j,k}$. It only remains to show that each $\mathcal U_{j,k}$ is dense. Fix $j,k\geq 1$ and a nonempty open set $V\subseteq G$, we want to show that $\mathcal U_{j,k}\cap V\neq\varnothing$. Fix a finite-order $c\in V$ and $\varepsilon>0$ such that $B_p(c,\varepsilon)\subseteq V$. The subgroup $A=\langle a_j,c\rangle$ is finite, so that $M=\max\{p(z)\mid z\in A\}$ is also finite. Choose $s\geq 1$ big enough to have $2^{-s}M<\varepsilon$ and set $d=a_j-2^sc$, which is a finite-order element of $G$ (since it belongs to $A$). By Corollary \ref{cor: dyadic roots} there exists a finite-order $u\in G$ with $$2^su=d\quad\text{ and }\quad p(u)=\frac1{2^s}p(d)\leq\frac M{2^s}<\varepsilon.$$
    Set $x=c+u$. We have that $p(c-x)=p(u)<\varepsilon$, so that $x\in V$, but also that $2^sx=2^sc+2^su=2^sc+d=a_j$, that is $x\in\mathcal U_{j,k}$, showing that $\mathcal U_{j,k}$ is dense.

    Since $\Gen(G)=\bigcap_{j,k}\mathcal U_{j,k}$ is a countable intersection of dense open sets, it is comeager in $G$ by the Baire category theorem, concluding the argument.
\end{proof}

Cameron and Vershik \cite{cameron2006isometry} showed that there exists a monothetic group structure on the Urysohn space $\mathbb U$ for which the metric of $\mathbb U$ is invariant, while a combination of \cite[Theorem 0.1]{doucha2019} and \cite[Theorem 6.4]{melleray2013generic} applied to $\mathbb Z$ shows that there exists a monothetic and extremely amenable group structure on $\mathbb U$ for which the metric of $\mathbb U$ is invariant. We extend those results by showing that the group structure can also be required to be strongly exotic and L\'{e}vy and by showing that they also hold for the Urysohn sphere $\mathbb U_1$.

\begin{corollary} \label{corollary: group structure on Urysohn sphere}
    Let $\mathbb U_\infty=\mathbb U$ be the Urysohn space and let $\mathbb U_1$ be the Urysohn sphere. For $r\in\{1,\infty\}$, there exists a monothetic, L\'{e}vy, strongly exotic group structure on $\mathbb U_r$ for which the metric of $\mathbb U_r$ is invariant.
\end{corollary}
\begin{proof}
    This follows immediately from Theorem \ref{thm: G_r(0) is generically monothetic}, Theorem \ref{thm: Niemiec groups are Levy}, Theorem \ref{theorem: G_r(N) is strongly exotic} and the fact that $\G_r(0)$ with its invariant metric is isometric to $\mathbb U_r$ \cite[Theorem 5.1]{Niemiec}.
\end{proof}
%\printbibliography

\subsection*{Acknowledgements}
Large language models (ChatGPT version Sol 5.6 and Claude version Fable 5) were used in producing the proofs presented in this paper, including the generation of substantial portions of the initial proofs. The authors then rewrote the paper independently and take full responsibility for the final formulation, exposition, and correctness of the mathematical content.

\bibliographystyle{amsplain}
\bibliography{main}

\end{document}